\documentclass[a4paper,10pt]{article}

\usepackage[greek,english]{babel}  
\usepackage{graphicx}
\usepackage{amsmath}
\usepackage{amssymb}
\usepackage{amsmath,amsthm}
\usepackage{verbatim}
\usepackage{sidecap}
\usepackage{tikz} 
\usepackage{float}
\usepackage{caption}
\usetikzlibrary{quotes,angles}
\usepackage{pgfplots}
\usepgfplotslibrary{groupplots}
\usepgflibrary{arrows}
\usetikzlibrary{calc}

\usepackage{subcaption}

\usepackage{tcolorbox}

\usepackage{geometry} 
\newtheorem{theorem}{Theorem}[section]
\newtheorem{definition}[theorem]{Definition}

\newtheorem{proposition}[theorem]{Proposition}

\theoremstyle{definition}
\newtheorem{remark}[theorem]{Remark}

\newcommand{\RR}{\mathbb{R}}

\newcommand{\ds}{\displaystyle}

\begin{document}

\title{{\bf On the minimality of Keplerian arcs with fixed negative energy\thanks{Work supported by the ERC Advanced Grant 2013 n.~339958 \textit{Complex Patterns for Strongly Interacting Dynamical Systems - COMPAT}. 
Work written under the auspices of the Grup\-po Na\-zio\-na\-le per l'Anali\-si Ma\-te\-ma\-ti\-ca, la Pro\-ba\-bi\-li\-t\`{a} e le lo\-ro Appli\-ca\-zio\-ni (GNAMPA) of the Isti\-tu\-to Na\-zio\-na\-le di Al\-ta Ma\-te\-ma\-ti\-ca (INdAM). In particular, A.B. acknowledges the support of the INdAM-GNAMPA Project ``Il modello di Born-Infeld per l'elettromagnetismo
nonlineare: esistenza, regolarit\`{a} e molteplicit\`{a} di soluzioni''.}}}

\author{
\vspace{1mm}\\
{\bf\large Vivina Barutello, Alberto Boscaggin, Walter Dambrosio}
\vspace{1mm}\\
{\it\small Dipartimento di Matematica, Universit\`a di Torino}\\
{\it\small via Carlo Alberto 10}, {\it\small 10123 Torino, Italy}\\
{\it\small e-mails: vivina.barutello@unito.it, alberto.boscaggin@unito.it, walter.dambrosio@unito.it}\vspace{1mm}}

\date{}

\maketitle

\centerline{\textit{\large{In memory of Florin Diacu}}}

\vspace{0.7cm}



\begin{abstract}

We revisit a classical result by Jacobi \cite{Ja} on the local minimality, as critical points of the corresponding energy functional,
of fixed-energy solutions of the Kepler equation joining two distinct points with the same distance from the origin.
Our proof relies on the Morse index theorem, together with a characterization of the conjugate points as points of geodesic bifurcation.
\end{abstract}


\section{Introduction}

Variational methods have proved to be a powerful tool in constructing solutions to various equations of Celestial Mechanics 
(see, among many others, \cite{BahRab89,BarTerVer14,BosDamTer17,Che,CheMon,FerTer,FuGrNe,MadVen09,SoaTer12,TeVe,Yu} and the references therein). Within this approach, as a matter of fact, a key role is played by the Kepler equation
\begin{equation}\label{eq_kep}
\ddot q(t) = - \frac{q(t)}{\vert q(t) \vert^3}, \qquad q \in \mathbb{R}^2 \setminus \{0\},
\end{equation}
to which one is typically led, via blow-up analysis or comparison arguments. Therefore, a solid knowledge of the variational properties of Keplerian orbits is often necessary.

In this paper, we are concerned with Keplerian orbits with energy $h < 0$ joining two points with the same distance from the origin. Precisely, we deal with the fixed-energy problem
\begin{equation} \label{eq-main1}
		\begin{cases}
		\ds \ddot{q}(t) = - \frac{q(t)}{|q(t)|^3}, \quad \forall t\in (-\omega',\omega'), \vspace{0.2cm}\\
        \ds q(-\omega') = p',\ q(\omega') = q',
				\vspace{0.2cm}\\
		\ds \dfrac{1}{2} |\dot{q}(t)|^2- \frac{1}{|q(t)|} = h, \quad \forall t\in [-\omega',\omega'],
	\end{cases}
	\end{equation}
where	$p', q' \in \RR^2 \setminus \{0\}$, with $|p'|=|q'|=R'$ and $p'\neq q'$. The energy $h < 0$ is fixed ($\omega'>0$ is unprescribed).
Notice that solutions of this problem (if any) have trajectories which lie on Keplerian ellipses (with energy $h$): in particular, they always appear in pairs, since any ellipse gives rise to two distinct arcs (one with positive angular momentum, one with negative angular momentum).
Preliminary,  in Section \ref{sec2} we recover existence, multiplicity and mutual positions of these Keplerian arcs, through a completely elementary geometrical analysis. 

From a variational point of view, solutions of problem \eqref{eq-main1}
can be interpreted as geodesic arcs in the punctured plane, when endowed with the Riemannian metric
$g(q) = (1/\vert q \vert+h)\langle \cdot, \cdot \rangle$ where $\langle \cdot, \cdot \rangle$
is the Euclidean product in $\mathbb{R}^2$. This is the content of the well known Maupertuis-Jacobi principle: solutions to \eqref{eq-main1}   correspond - after a suitable reparameterization - to critical points of the so-called energy functional
$$
E_h(q) = \int_0^1 \vert \dot q(t) \vert^2 \left( \frac{1}{\vert q(t) \vert} + h\right)\,dt,
$$
(see Proposition \ref{prop_maup}).
It is therefore natural to investigate minimality properties of such Keplerian geodesics.
A first answer to this problem was already given by Jacobi in 1837 in the paper \cite{Ja}; later Todhunter translated Jacobi's work in \cite{Tod} (see section 226) and Winter synthesized it in his book \cite{Wint} (sections from 247 to 254). These authors developed a geometrical analysis based on localization of conjugate points on Keplerian geodesics (see also the more recent book by Gutzwiller \cite{Gut90}, section 2.5).

The aim of our paper is to revisit from a more modern perspective these classical results. In particular, a key role in our proof is played by a  characterization of conjugate points as points of geodesic bifurcation, according to the theory recently developed in \cite{PicPorTau04}. For the reader's convenience, a brief recap on these arguments is provided in the final Appendix. A precise statement of our main result is given in Section \ref{sec3} (see Theorem \ref{teo-var}), where a complete description of the minimality of Keplerian arcs, depending on the angular momentum and on their mutual position, is provided.
 
A complementary result for the case $h>0$ has been proposed recently by Montgomery in \cite{Mont}.

Let us finally mention that a whole Keplerian ellipse, when regarded as a closed geodesic for the corresponding energy functional, 
is not a local minimizer. This is in striking contrast with the celebrated paper \cite{Gor77} by Gordon, characterizing, for any given $T > 0$, the Keplerian ellipses of minimal period $T$ as minima of the action functional
$$
A_T(q) = \int_0^T \left( \frac12 \vert \dot q(t) \vert^2 + \frac{1}{\vert q(t) \vert}\right)\,dt
$$ 
on the set of non-contractible $H^1$-loops in the punctured plane (see, for instance, \cite{BaJaPo,HuSun10,KOP} and the references therein for some other contributions on this line of research).

\section{The existence and multiplicity result}\label{sec2}

We first observe that through the scaling
\[
x(t)=\vert h \vert q\left(\dfrac{t}{\vert h \vert^{3/2}}\right), \qquad t\in [-\vert h \vert^{3/2}\omega',\vert h \vert^{3/2}\omega'],
\] 
solutions of \eqref{eq-main1} correspond to solutions of 
\begin{equation} \label{eq-main}
\begin{cases}
\ds \ddot{x}(t) = - \frac{x(t)}{|x(t)|^3}, \quad \forall t\in (-\omega,\omega), \vspace{0.2cm}\\
\ds x(-\omega) = p,\ x(\omega)=q, \vspace{0.2cm} \\
\ds \dfrac{1}{2} |\dot{x}(t)|^2- \frac{1}{|x(t)|} = -1, \quad \forall t\in [-\omega,\omega],
\end{cases}
\end{equation}
where $p=\vert h \vert p'$ and $q=\vert h \vert q'$ (again, $\omega = \vert h \vert^{3/2}\omega'$ is unprescribed). 
Therefore, in the following we will focus on problem \eqref{eq-main}. 

Let us also notice that, due to the energy relation, trajectories of solutions of \eqref{eq-main} are confined in the set
\begin{equation}\label{eq_hill}
\mathcal{H} = \left\{ x \in \mathbb{R}^2 \setminus \{0\}: \, \vert x \vert < 1\right\},
\end{equation}
the so-called \textit{Hill's region}. On the other hand, as stated in Theorem \ref{teo-ex}, a crucial role is played by the ball $B_{1/2}$, centered at the origin and having radius $1/2$. Indeed, the qualitative description of solutions of \eqref{eq-main} changes according if the end-points belong to $B_{1/2}$ or not. 

In order to state the existence result, for every solution $x$ of \eqref{eq-main} we need to introduce the (conserved) angular momentum $x\wedge \dot{x}$; using polar coordinates centered at the origin of the plane of motion, $x=re^{i\phi}$ (with $r > 0$), it writes as
\[
c_x e^{i\phi}\wedge ie^{i\phi},
\]
where $c_x=r^2\dot{\phi}$. Non-collision orbits connecting $p$ and $q$ can be classified according to the sign of $c_x$;  to this end we introduce the sets
\[
{\cal C}^+=\{x\in C^2 ([-\omega,\omega];\RR^2\setminus \{0\}):\ c_x>0\},\quad {\cal C}^-=\{x\in C^2 ([-\omega,\omega];\RR^2\setminus \{0\}):\ c_x<0\}.
\]
Moreover, we denote by $2\phi_0\in (0,2\pi)$ the counterclockwise angle bewteen $0p$ and $0q$.

We then have the following result.

\begin{theorem} \label{teo-ex}
The following hold true:
\begin{enumerate}
\item[(i)] if $|p|=|q|<1/2$, then for every $\phi_0 \in (0,\pi)$ there exist four solutions $x_{\mathrm{int}}^\pm$ and $x_{\mathrm{ext}}^\pm$ of \eqref{eq-main} such that
\begin{itemize}
\item[-] $x_{\mathrm{int}}^\pm \in {\cal C}^\pm$ and $x_{\mathrm{ext}}^\pm \in {\cal C}^\pm$,
\item[-] $|x_{\mathrm{int}}^\pm (t)| < |p|$ and $|x_{\mathrm{ext}}^\pm (t)| > |p|$, for every $t\in (-\omega,\omega)$;
\end{itemize}
\item[(ii)] if $|p|=|q|=1/2$, then for every $\phi_0 \in (0,\pi) \setminus \{\pi/2\}$ there exist four solutions $x^\pm$, $x_{\mathrm{int}}$ and $x_{\mathrm{ext}}$ of \eqref{eq-main} such that
\begin{itemize}
\item[-] $x^\pm \in {\cal C}^\pm$; $x_{\mathrm{ext}}\in {\cal C}^+$, $x_{\mathrm{int}}\in {\cal C}^-$ if $\phi_0 \in (0,\pi/2)$,
$x_{\mathrm{ext}}\in {\cal C}^-$, $x_{\mathrm{int}}\in {\cal C}^+$ if $\phi_0 \in (\pi/2,\pi)$, 
\item[-] $|x^\pm(t)|=|p|$, $|x_{\mathrm{int}}(t)| < |p|$, $|x_{\mathrm{ext}}(t)| > |p|$, for every $t\in (-\omega,\omega)$;
\end{itemize}
if $\phi_0=\pi/2$ there exist two solutions $x^{\pm} \in {\cal C}^\pm$ such that $|x^\pm(t)|=|p|$ , for every $t\in (-\omega,\omega)$;
\item[(iii)] if $1/2<|p|=|q|<1$, then \eqref{eq-main} admits solutions only if $\phi_0 \in \left(0,\arcsin \frac{1-|p|}{|p|}\right) \cup 
\left(\pi-\arcsin \frac{1-|p|}{|p|},\pi\right)$. Moreover: if $\phi_0 \in\left(0,\arcsin \frac{1-|p|}{|p|}\right)$, there exist four solutions $x_{\mathrm{int},1}^-$, $x_{\mathrm{int},2}^-$, $x_{\mathrm{ext},1}^+$ and $x_{\mathrm{ext},2}^+$ of \eqref{eq-main} such that
\begin{itemize}
	\item[-] $x_{\mathrm{int},i}^- \in {\cal C}^-$ and $x_{\mathrm{ext},i}^+ \in {\cal C}^+$, $i=1, 2$,
	\item[-] $|x_{\mathrm{int},i}^- (t)| < |p|$ and $|x_{\mathrm{ext},i}^+ (t)| >|p|$, for every $t\in (-\omega,\omega)$, $i=1, 2$,
	\item[-] $\vert x_{\mathrm{ext},1}^+(0) \vert < \vert x_{\mathrm{ext},2}^+(0) \vert$ and $\vert x_{\mathrm{int},1}^-(0) \vert > \vert x_{\mathrm{int},2}^-(0) \vert$;
\end{itemize}
if $\phi_0 \in \left(\pi-\arcsin \frac{1-|p|}{|p|},\pi\right)$, there exist four solutions $x_{\mathrm{int},1}^+$, $x_{\mathrm{int},2}^+$, $x_{\mathrm{ext},1}^-$ and $x_{\mathrm{ext},2}^-$ of \eqref{eq-main} such that
\begin{itemize}
	\item[-] $x_{\mathrm{int},i}^+ \in {\cal C}^+$ and $x_{\mathrm{ext},i}^- \in {\cal C}^-$, $i=1, 2$,
	\item[-] $|x_{\mathrm{int},i}^+ (t)| < |p|$ and $|x_{\mathrm{ext},i}^- (t)| >|p|$, for every $t\in (-\omega,\omega)$, $i=1, 2$,
	\item[-] $\vert x_{\mathrm{ext},1}^-(0) \vert < \vert x_{\mathrm{ext},2}^-(0) \vert$ and $\vert x_{\mathrm{int},1}^+(0) \vert > \vert x_{\mathrm{int},2}^+(0) \vert$.
\end{itemize}
\end{enumerate}
\end{theorem}

\begin{remark}
As it will be clear from the proof, the solutions given in Theorem \ref{teo-ex} are the only self-intersection free solutions of \eqref{eq-main}. 
\end{remark}

A picture of the trajectories of the solutions to \eqref{eq-main} in the three cases is drawn in Figures \ref{fig-dentrocircolare}, \ref{fig-circolare} and \ref{fig-fuoricircolare}.

	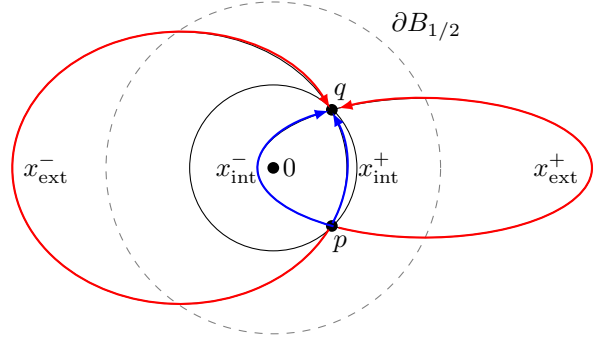
\begin{figure}[H]
	\centering
	\begin{tikzpicture}[scale=1.1]
	\draw (0,0) circle (1cm);
	\draw[dashed,gray] (0,0) circle (2cm);
	\draw (1.81,0) ellipse (2cm and 0.84cm);
	\draw (-1.12,0) ellipse (2cm and 1.64cm);
	\node[right] at (0,0) {$0$};
	\fill (0.7,0.7) circle[radius=2pt];
	\fill (0,0)  circle[radius=2pt];
	\fill (0.7,-0.7)  circle[radius=2pt];
	\node[below] at (0.8,-0.7) {$p$};
	\node[above] at (0.8,0.7) {$q$};
	\node[right] at (1.3,1.7) {$\partial B_{1/2}$};
	\node[right] at (-3.1,0) {$x_{\mathrm{ext}}^-$};
	\node[right] at (-0.8,0) {$x_{\mathrm{int}}^-$};
	\node[right] at (0.9,0) {$x_{\mathrm{int}}^+$};
	\node[right] at (3,0) {$x_{\mathrm{ext}}^+$};
	\coordinate (Q) at ($(-1.12,0) + (-24:2cm and 1.64cm)$);
	\draw[thick, blue,-latex] ($(-1.12,0) + (-24:2cm and 1.64cm)$(Q) arc (-24:24:2cm and 1.64cm);
	\coordinate (QQ) at ($(-1.12,0) + (26:2cm and 1.64cm)$);
	\draw[thick, red,latex-] ($(-1.12,0) + (26:2cm and 1.64cm)$(QQ) arc (26:333:2cm and 1.64cm);
	\coordinate (P) at ($(1.81,0) + (-121:2cm and 0.84cm)$);
	\draw[thick, red,-latex] ($(1.81,0) + (-121:2cm and 0.84cm)$(P) arc (-121:121:2cm and 0.84cm);
	%
	\draw[thick, blue,latex-] ($(1.81,0) + (125:2cm and 0.84cm)$(P) arc (125:237:2cm and 0.84cm);
	\end{tikzpicture}
	\caption{The four trajectories in the case $|p|=|q|<1/2$, as described in Theorem \ref{teo-ex} (i): two of them, $x_{\mathrm{int}}^\pm$, lie inside the circle of radius $|p|$ and have opposite angular momentum; the other two, $x_{\mathrm{out}}^\pm$, still with opposite angular momentum, lie outside the circle of radius $|p|$. Notice that such trajectories exist for any choice of $p \in \partial B_{|p|}$.}
	\label{fig-dentrocircolare}
\end{figure}

	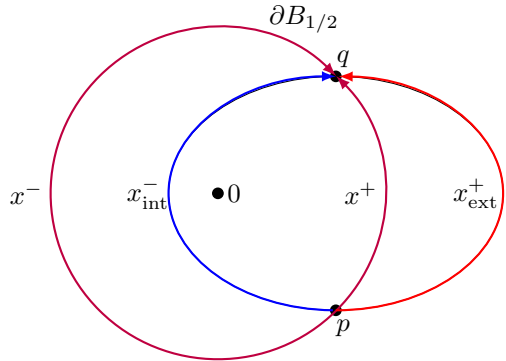
\begin{figure}[H]
	\centering
	\begin{tikzpicture}[scale=1.1]
	\draw (1.41,0) ellipse (2cm and 1.41cm);
	\node[right] at (0,0) {$0$};	
	\fill (1.41,1.41) circle[radius=2pt];
	\fill (0,0)  circle[radius=2pt];
	\fill (1.41,-1.41)  circle[radius=2pt];
	\node[below] at (1.5,-1.41) {$p$};
	\node[above] at (1.5,1.41) {$q$};
	\node[right] at (-2.6,0) {$x^-$};
	\node[right] at (-1.2,0) {$x_{\mathrm{int}}^-$};
	\node[right] at (1.4,0) {$x^+$};
	\node[right] at (2.7,0) {$x_{\mathrm{ext}}^+$};
	\node[right] at (0.5,2.1) {$\partial B_{1/2}$};
	\coordinate (Q) at ($(0,0) + (-45:2cm and 2cm)$);
	\draw[thick,purple,-latex] ($(0,0) + (-45:2cm and 2cm)$(Q) arc (-45:45:2cm and 2cm);
	%
	\draw[thick,purple,-latex] ($(0,0) + (-45:2cm and 2cm)$(Q) arc (-45:-315:2cm and 2cm);
	\coordinate (P) at ($(1.41,0) + (-90:2cm and 1.41cm)$);
    \draw[thick, red,-latex]($(1.41,0) + (-90:2cm and 1.41cm)$(P) arc (-90:89:2cm and 1.41cm);
	%
	\draw[thick, blue,latex-] ($(1.41,0) + (90:2cm and 1.41cm)$(P) arc (90:269:2cm and 1.41cm);
	\end{tikzpicture}
	\caption{The four trajectories in the case $|p|=|q|=1/2$, as described in Theorem \ref{teo-ex} (ii): here one of the Keplerian ellipses is just the circumference $\partial B_{1/2}$, giving rise to the solutions $x^{\pm}$; the other two solutions are denoted by $x_{\mathrm{int}}^-$ and $x_{\mathrm{ext}}^+$, lying inside and outside the circle of radius $1/2$, respectively. Let us note that if $p$ and $q$ are antipodal, the picture degenerates and only two solutions exist: roughly, $x_{\mathrm{int}}^-=x^-$ and $x_{\mathrm{ext}}^+ = x^+$.}
	\label{fig-circolare}
\end{figure}

	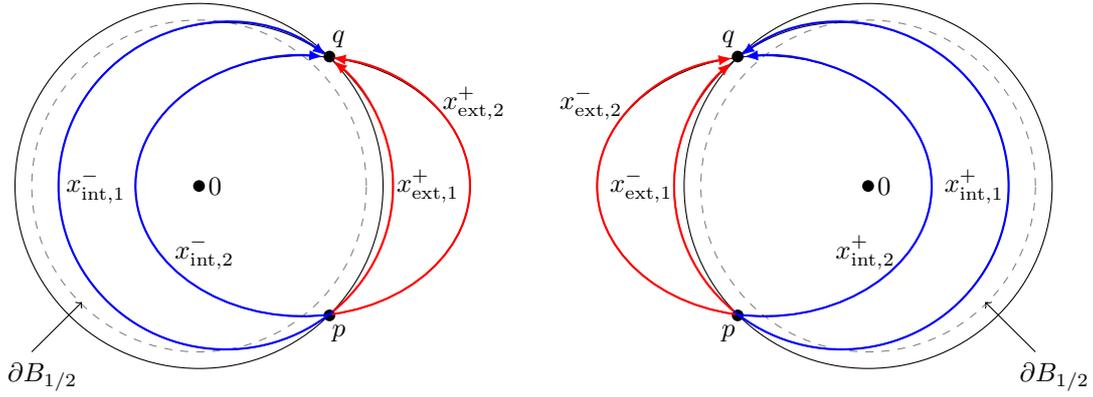
\begin{figure}[H]
	\centering
	\begin{tikzpicture}[scale=1.1]
	\draw[dashed,gray] (0,0) circle (2cm);
	\draw (0,0) circle (2.2cm);
	\draw[dashed,gray] (8,0) circle (2cm);
	\draw (8,0) circle (2.2cm);
	%
	\draw (1.24,0) ellipse (2cm and 1.57cm);
	\draw (0.32,0) ellipse (2cm and 1.97cm);
	\draw (8-1.24,0) ellipse (2cm and 1.57cm);
	\draw (8-0.32,0) ellipse (2cm and 1.97cm);
	\fill (1.56,1.56) circle[radius=2pt];
	\fill (0,0)  circle[radius=2pt];
	\fill (1.56,-1.56)  circle[radius=2pt];
	\fill (8-1.56,1.56) circle[radius=2pt];
	\fill (8,0)  circle[radius=2pt];
	\fill (8-1.56,-1.56)  circle[radius=2pt];
	\node[right] at (0,0) {$0$};	
	\node[below] at (1.67,-1.56) {$p$};
	\node[above] at (1.67,1.56) {$q$};
	\node[right] at (8,0) {$0$};	
	\node[below] at (8-1.67,-1.56) {$p$};
	\node[above] at (8-1.67,1.56) {$q$};
	\node[right] at (-2.4,-2.3) {$\partial B_{1/2}$};
	\node[right] at (9.7,-2.3) {$\partial B_{1/2}$};
	\draw[->] (-2,-2)->(-1.4,-1.4);
	\draw[->] (10,-2)->(9.4,-1.4);
	\node[right] at (-1.7,0) {$x_{\mathrm{int},1}^-$};
	\node[right] at (-0.4,-0.8) {$x_{\mathrm{int},2}^-$};
	\node[right] at (2.25,0) {$x_{\mathrm{ext},1}^+$};
	\node[right] at (2.8,1) {$x_{\mathrm{ext},2}^+$};
	\node[right] at (8.8,0) {$x_{\mathrm{int},1}^+$};
	\node[right] at (7.5,-0.8) {$x_{\mathrm{int},2}^+$};
	\node[right] at (4.8,0) {$x_{\mathrm{ext},1}^-$};
	\node[right] at (4.2,1) {$x_{\mathrm{ext},2}^-$};
	\coordinate (Q) at ($(1.24,0) + (-80:2cm and 1.57cm)$);
	\draw[thick, red,-latex] ($(1.24,0) + (-80:2cm and 1.57cm)$(Q) arc (-80:80:2cm and 1.57cm);
	\draw[thick, blue,-latex] ($(1.24,0) + (-80:2cm and 1.57cm)$(Q) arc (-80:-277:2cm and 1.57cm);
	\coordinate (P) at ($(0.32,0) + (-50:2cm and 1.97cm)$);
	\draw[thick, red,-latex] ($(0.32,0) + (-50:2cm and 1.97cm)$(P) arc (-50:50:2cm and 1.97cm);
	\draw[thick, blue,-latex] ($(0.32,0) + (-50:2cm and 1.97cm)$(P) arc	(-50:-307:2cm and 1.97cm);
	\coordinate (QQ) at ($(8-1.24,0) + (-100:2cm and 1.57cm)$);
	\coordinate (PP) at ($(0.32,0) + (-50:2cm and 1.97cm)$);
	\draw[thick, red,-latex] ($(8-1.24,0) + (-100:2cm and 1.57cm)$(QQ) arc (-100:-259:2cm and 1.57cm);
	\draw[thick, red,-latex] ($(8-0.32,0) + (-130:2cm and 1.97cm)$(PP) arc (-130:-229:2cm and 1.97cm);
	\draw[thick, blue,-latex] ($(8-1.24,0) + (-100:2cm and 1.57cm)$(QQ) arc (-100:97:2cm and 1.57cm);
	\draw[thick, blue,-latex] ($(8-0.32,0) + (-130:2cm and 1.97cm)$(PP) arc (-130:127:2cm and 1.97cm);
\end{tikzpicture}
	\caption{The four trajectories in the case $|p|=|q| \in (1/2,1)$, as described in Theorem \ref{teo-ex} (iii). Notice that, differently from the case (i), the two solutions lying outside the circle of radius $\vert p \vert$ have the same angular momentum; the same is true for the solutions lying inside the circle. In order to label them, we further need to distinguish the case $\phi_0 \in \left(0,\arcsin \frac{1-|p|}{|p|}\right)$ and $\phi_0 \in \left(\pi-\arcsin \frac{1-|p|}{|p|},\pi\right)$: in the first one (on the left), external (resp., internal) solutions have positive (resp., negative) angular momentum, in the second one (on the right), the picture is reversed. When $\phi_0 \in \left[\arcsin \frac{1-|p|}{|p|},\pi-\arcsin \frac{1-|p|}{|p|}\right]$, solutions do not exist at all.}
	\label{fig-fuoricircolare}
\end{figure}

The proof of Theorem \ref{teo-ex} relies on a purely geometrical argument, based on the well-known fact (see, for instance, \cite{Ge16, OrUr10}) that the trajectories of solutions of \eqref{eq-main} lie on ellipses with one focus at the origin and major axis of lenght $1$; using polar coordinates $(r,\phi)$, such kind of ellipses can be parametrized as
\begin{equation} \label{eq-ellpolari}
r=\dfrac{0.5(1-e^2)}{1+e\cos (\phi+\phi')},\qquad \phi \in [0,2\pi],
\end{equation}
where $e$ denotes the eccentricity of the ellipse and $(\cos \phi',\sin \phi')$ is the opposite of the direction of the vector connecting the origin and the second focus of the ellipse, see Figure \ref{figell}.

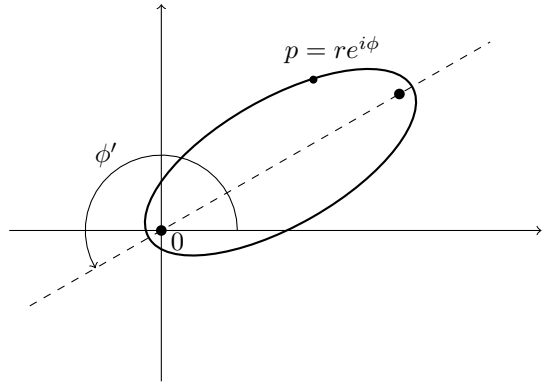
\begin{figure}[H]
	\centering
	\begin{tikzpicture}[scale=1]
	\draw[->] (-2,0) -- (5,0);
	\draw[->] (0,-2) -- (0,3);
	\draw[rotate=30,thick] (1.81,0) ellipse (2cm and 0.84cm);
	\draw[dashed] (-1.73,-1) -- (2.5*1.73,2.5*1);
	\draw[domain=0:210,->] plot ({cos(\x)},{sin(\x)});
	\node[right] at (-1,1) {$\phi'$};
	\fill (0,0)  circle[radius=2pt];
	\node[right] at (0,-0.15) {$0$};
	\fill (1.81*1.73,1.81)  circle[radius=2pt];
	\fill (2,2)  circle[radius=1.5pt];
	\node[right] at (1.5,2.4) {$p=r e^{i\phi}$};
	\end{tikzpicture}
	\caption{The polar equation of an ellipse as in formula \eqref{eq-ellpolari}: notice that $\phi'$ is the angle of the perihelion of the ellipse.}
	\label{figell}
\end{figure}

Due to rotational invariance of the problem, it is not restrictive to assume that the fixed end-points are symmetric with respect to the $x$-axis and that the starting point $p$ belongs to the lower half-plane; as a consequence the admissible values for $\phi'$ in \eqref{eq-ellpolari} are $\phi'=0$ or $\phi'=\pi$, corresponding to an ellipse with second focus on the negative or positive $x$-semiaxis, respectively.

Therefore, the admissible equations of the trajectories are 
\begin{equation} \label{eq-sxellpolari}
r=\dfrac{0.5(1-e^2)}{1+e\cos \phi},\qquad \phi \in [0,2\pi],
\end{equation}
and
\begin{equation} \label{eq-dxellpolari}
r=\dfrac{0.5(1-e^2)}{1-e\cos \phi},\qquad \phi \in [0,2\pi].
\end{equation}

With this in mind, we can give the proof of Theorem \ref{teo-ex}.

\begin{proof}[Proof of Theorem \ref{teo-ex}.]

We distinguish the three cases.
\smallbreak
\noindent
(i) Assume that $|p|=|q|<1/2$ and $\phi_0\in (0,\pi)$. According to the previous discussion, we first need to prove that there exist two ellipses of the form \eqref{eq-sxellpolari} or \eqref{eq-dxellpolari} connecting $p$ and $q$; indeed, we show that there exists one ellipse having equation of the type \eqref{eq-sxellpolari} and one of the type \eqref{eq-dxellpolari} (cf. again Figure \ref{fig-dentrocircolare}).

To this end, observing that, by the definition of $\phi_0$, the polar coordinates of $q$ are $q=(|q|,\phi_0)$, it is immediate to check that an ellipse of the form \eqref{eq-sxellpolari} passes from $q$ if and only if 
\begin{equation} \label{eq-eccelldentro1}
e=e_1=-|q|\cos \phi_0+\sqrt{|q|^2\cos^2 \phi_0+1-2|q|}
\end{equation}
while an ellipse of the form \eqref{eq-dxellpolari} passes from $q$ if and only if 
\begin{equation} \label{eq-eccelldentro2}
e=e_2=|q|\cos \phi_0+\sqrt{|q|^2\cos^2 \phi_0+1-2|q|}.
\end{equation}
We denote by ${\cal E}_1$ and ${\cal E}_2$ the ellipses of eccentricities \eqref{eq-eccelldentro1} and \eqref{eq-eccelldentro2}, respectively; it it clear that the points $p$ and $q$ divide each ellipse in two arcs, which are trajectories of solutions of \eqref{eq-main}. 

We conclude the proof in the case when $\phi_0\in (\pi/2,\pi)$, the other case being similar.
Let us denote by $x_{\mathrm{ext}}^-$ the arc of ${\cal E}_1$ joining $p$ and $q$ in the clockwise sense; $x_{\mathrm{ext}}^-$ is the image of a solution of \eqref{eq-main} which we still denote by $x_{\mathrm{ext}}^-$, with an abuse of notation. Analogously, let $x_{\mathrm{int}}^+$ be the arc of ${\cal E}_1$ joining $p$ and $q$ in the anticlockwise sense and let also $x_{\mathrm{int}}^+$ be the corresponding solution of \eqref{eq-main}. 

We claim that functions $x_{\mathrm{ext}}^-$ and $x_{\mathrm{int}}^+$ satisfy the desired properties.
Indeed, $x_{\mathrm{ext}}^-\in {\cal C}^-$ and $x_{\mathrm{int}}^+\in {\cal C}^+$, by construction; moreover, recalling \eqref{eq-sxellpolari}, $x_{\mathrm{ext}}^-$ can be parametrized as
\[
r(t)=\dfrac{0.5(1-e^2)}{1+e\cos \phi (t)},\qquad t\in [-\omega,\omega],
\]
with $\phi(t)\in [\phi_0,2\pi-\phi_0]$, for every $t\in [-\omega,\omega]$. Hence,
\[
\cos \phi(t)\leq \cos \phi_0,\quad \forall \ t\in [-\omega,\omega],
\]
and then
\[
|x_{\mathrm{ext}}^-(t)|=r(t)\geq \dfrac{0.5(1-e^2)}{1+e\cos \phi_0}=|p|,\quad \forall \ t\in [-\omega,\omega].
\]
In an analogous way, recalling \eqref{eq-dxellpolari}, $x_{\mathrm{int}}^+$ can be parametrized as
\[
r(t)=\dfrac{0.5(1-e^2)}{1+e\cos \phi (t)},\qquad t\in [-\omega,\omega],
\]
with $\phi(t)\in [-\phi_0,\phi_0]$, for every $t\in [-\omega,\omega]$. Hence,
\[
\cos \phi(t)\geq \cos \phi_0,\quad \forall \ t\in [-\omega,\omega],
\]
and then
\[
|x_{\mathrm{int}}^+(t)|=r(t)\leq \dfrac{0.5(1-e^2)}{1+e\cos \phi_0}=|p|,\quad \forall \ t\in [-\omega,\omega].
\]
Arguing in a similar way with the ellipse ${\cal E}_2$ it is possible to prove the existence of the other two solutions of \eqref{eq-main} with the stated properties.

\smallbreak
\noindent
(ii) Assume now that $|p|=|q|=1/2$ and $\phi_0\in (0,\pi)$. In this situation there exist the circular orbit $r=1/2$ and an elliptic orbit of the type
\begin{equation} \label{eq-ellissecirc1}
r=\dfrac{0.5(1-e^2)}{1-e\cos \phi},\qquad \phi \in [0,2\pi],
\end{equation}
with $e=\cos \phi_0$, if $\phi_0\in (0,\pi/2]$, and of the type
$$
r=\dfrac{0.5(1-e^2)}{1+e\cos \phi},\qquad \phi \in [0,2\pi],
$$
with $e=-\cos \phi_0$ if $\phi_0\in (\pi/2,\pi)$ (cf. Figure \ref{fig-circolare}). An argument similar to the one developed in the previous case leads to the conclusion.

\smallbreak
\noindent
(iii) Assume now that $1/2<|p|=|q|<1$; a simple geometrical argument shows that it is possible to join $p$ and $q$ with a Keplerian ellipse of major axis of lenght $1$ if and only if
\[
|p|+|p|\sin \phi_0 <1,
\]
i.e. 
\[
\sin \phi_0 <\dfrac{1-|p|}{|p|}.
\]
Let us remark that the right-side of the inequality is a quantity in $(0,1)$, hence we indeed have a restriction on $\phi_0$.
Under this condition, if $\phi_0\in (\pi-\arcsin\frac{1-|p|}{|p|},\pi)$ there exist two elliptic orbits of the type
$$
r=\dfrac{0.5(1-e^2)}{1+e\cos \phi},\qquad \phi \in [0,2\pi],
$$
with $e=-|p|\cos \phi_0\pm \sqrt{|p|^2\cos^2 \phi_0+1-2|p|}$.
On the other hand, if $\phi_0 \in (0,\arcsin \frac{1-|p|}{|p|})$ there exist two elliptic orbits of the type
\begin{equation} \label{eq-ellissefuoricirc2}
r=\dfrac{0.5(1-e^2)}{1-e\cos \phi},\qquad \phi \in [0,2\pi],
\end{equation}
with $e=|p|\cos \phi_0\pm \sqrt{|p|^2\cos^2 \phi_0+1-2|p|}$ (cf. Figure \ref{fig-fuoricircolare}).
The existence of the four solutions of \eqref{eq-main} in the two situations follows from the existence of the Keplerian ellipses arguing as in the first case of the proof.
\end{proof}

\section{The variational characterization}\label{sec3}

Let us first clarify the variational formulation of the fixed-energy problem \eqref{eq-main}.
Consider the Riemannian manifold $(\mathcal{H},g)$, where $\mathcal{H}$ is the Hill's region defined in \eqref{eq_hill} and
the Riemannian metric is given by 
$$
g(x)[u,v] = \left( \frac{1}{\vert x \vert} - 1\right)\langle u,v\rangle, \qquad x \in \mathcal{H}, \, u,v \in \mathbb{R}^2 ,
$$
where $\langle \cdot, \cdot \rangle$ denotes the Euclidean product on $\mathbb{R}^2$. The following proposition, which is a version of the well-known Maupertuis-Jacobi principle, holds true.

\begin{proposition}\label{prop_maup}
Let $x: [-\omega,\omega] \to \mathbb{R}^2$ be a solution of \eqref{eq-main}. Then, there exists a diffeomorphism $t: [0,1] \to [-\omega,\omega]$ such that the function
$$
\gamma(s) = x(t(s)), \qquad s \in [0,1],
$$
is a geodesic in the Riemannian manifold $(\mathcal{H},g)$.
\end{proposition}

\begin{proof}[Sketch of the proof]
Defining, for $t \in [-\omega,\omega]$,
\begin{equation}\label{cambiost}
s(t) =  \frac{1}{S_x} \int_{-\omega}^t \left( \frac{1}{\vert x(\tau) \vert} - 1 \right)\,d\tau, \qquad \mbox{ where } S_x = \int_{-\omega}^\omega \left( \frac{1}{\vert x(\tau) \vert} - 1 \right)\,d\tau,
\end{equation}
and $t(s)$ as the functional inverse of $s(t)$, the conclusion follows from straightforward computations.
\end{proof}

\begin{remark}\label{rem_geoellisse}
The converse of Proposition \ref{prop_maup} also holds true: namely, any geodesic in $(\mathcal{H},g)$ admits a reparameterization
making it a solution of the Kepler equation with energy equal to $h=-1$. As a consequence, the images of geodesics in $(\mathcal{H},g)$ are arcs of ellipses with one focus at the origin and major axis of length $1$ (compare with \eqref{eq-ellpolari}).
\end{remark}

According to Proposition \ref{prop_energia}, a geodesic in the Riemannian manifold $(\mathcal{H},g)$, $\gamma$, is a critical point of the functional
$$
E(\gamma) = \int_0^1 \vert \dot \gamma(s) \vert^2 \left( \frac{1}{\vert \gamma(s) \vert} - 1\right)\,ds
$$
on the Hilbert manifold of $H^1$-paths $\gamma : [0,1] \to \mathcal{H}$ such that $\gamma(0) = p$ and $\gamma(1) = q$.
The local minimality/non-minimality of the corresponding $x$ as a solution of \eqref{eq-main}, as in Proposition \ref{prop_maup}, will thus be meant as the local minimality/non-miminality of the geodesic $\gamma$ for the energy functional $E$. 

With this in mind, our result reads as follows.

\begin{theorem} \label{teo-var}
With reference to Theorem \ref{teo-ex}, the following hold true:
\begin{enumerate}
\item if $|p|=|q|<1/2$, the solutions $x_{\mathrm{int}}^\pm$ are local minimizers while the solutions $x_{\mathrm{ext}}^\pm$ are not local minimizers;
\item if $|p|=|q|=1/2$, then:
\begin{itemize}
\item if $\phi_0 \in (0,\pi/2)$ the solutions $x^+$ and $x_{\mathrm{int}}$ are local minimizers while the solutions $x^-$ and $x_{\mathrm{ext}}$ are not local minimizers,
\item if $\phi_0 \in (\pi/2,\pi)$ the solutions $x^-$ and $x_{\mathrm{int}}$ are local minimizers while the solutions $x^+$ and $x_{\mathrm{ext}}$ are not local minimizers;
\end{itemize}
\item if $1/2<|p|=|q|<1$, then:
\begin{itemize}
\item if $\phi_0 \in\left(0,\arcsin \frac{1-|p|}{|p|}\right)$ the solutions $x_{\mathrm{ext},1}^+$ and $x_{\mathrm{int},2}^-$
are local minimizers while the solutions $x_{\mathrm{ext},2}^+$ and $x_{\mathrm{int},1}^-$ are not local minimizers,
\item if $\phi_0 \in \left(\pi-\arcsin \frac{1-|p|}{|p|},\pi\right)$ the solutions $x_{\mathrm{ext},1}^-$ and $x_{\mathrm{int},2}^+$
are local minimizers while the solutions $x_{\mathrm{ext},2}^-$ and $x_{\mathrm{int},1}^+$ are not local minimizers.
\end{itemize}
\end{enumerate}
\end{theorem}

A visual explanation of Theorem \ref{teo-var} is given in the figure below: compared with Figures \ref{fig-dentrocircolare}, \ref{fig-circolare} and \ref{fig-fuoricircolare}, here the locally minimal solutions are painted in magenta color.

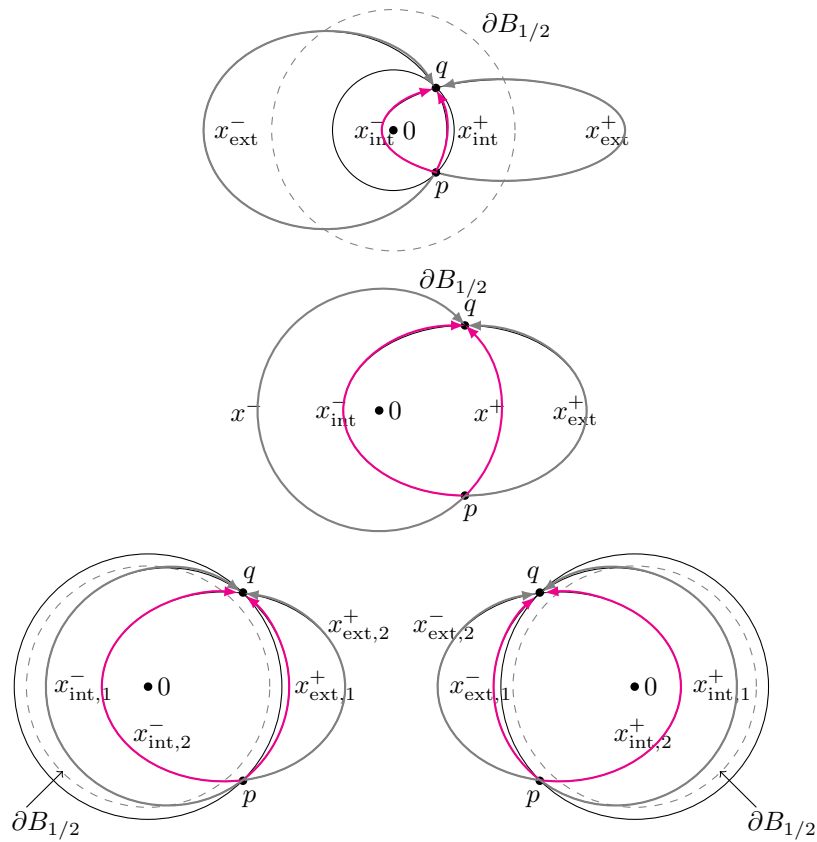
\begin{table}[H]
\centering
\begin{tabular}{c}
	\begin{tikzpicture}[scale=0.8]
	\draw (0,0) circle (1cm);
	\draw[dashed,gray] (0,0) circle (2cm);
	\draw (1.81,0) ellipse (2cm and 0.84cm);
	\draw (-1.12,0) ellipse (2cm and 1.64cm);
	\node[right] at (0,0) {$0$};
	\fill (0.7,0.7) circle[radius=2pt];
	\fill (0,0)  circle[radius=2pt];
	\fill (0.7,-0.7)  circle[radius=2pt];
	\node[below] at (0.8,-0.7) {$p$};
	\node[above] at (0.8,0.7) {$q$};
	\node[right] at (1.3,1.7) {$\partial B_{1/2}$};
	\node[right] at (-3.1,0) {$x_{\mathrm{ext}}^-$};
	\node[right] at (-0.8,0) {$x_{\mathrm{int}}^-$};
	\node[right] at (0.9,0) {$x_{\mathrm{int}}^+$};
	\node[right] at (3,0) {$x_{\mathrm{ext}}^+$};
	\coordinate (Q) at ($(-1.12,0) + (-24:2cm and 1.64cm)$);
	\draw[thick, magenta,-latex] ($(-1.12,0) + (-24:2cm and 1.64cm)$(Q) arc (-24:24:2cm and 1.64cm);
	\coordinate (QQ) at ($(-1.12,0) + (26:2cm and 1.64cm)$);
	\draw[thick,gray,latex-] ($(-1.12,0) + (26:2cm and 1.64cm)$(QQ) arc (26:333:2cm and 1.64cm);
	\coordinate (P) at ($(1.81,0) + (-121:2cm and 0.84cm)$);
	\draw[thick,gray,-latex] ($(1.81,0) + (-121:2cm and 0.84cm)$(P) arc (-121:121:2cm and 0.84cm);
	%
	\draw[thick, magenta,latex-] ($(1.81,0) + (125:2cm and 0.84cm)$(P) arc (125:237:2cm and 0.84cm);
	\end{tikzpicture}
  \\
	\begin{tikzpicture}[scale=0.8]
	\draw (1.41,0) ellipse (2cm and 1.41cm);
	\node[right] at (0,0) {$0$};	
	\fill (1.41,1.41) circle[radius=2pt];
	\fill (0,0)  circle[radius=2pt];
	\fill (1.41,-1.41)  circle[radius=2pt];
	\node[below] at (1.5,-1.41) {$p$};
	\node[above] at (1.5,1.41) {$q$};
	\node[right] at (-2.6,0) {$x^-$};
	\node[right] at (-1.2,0) {$x_{\mathrm{int}}^-$};
	\node[right] at (1.4,0) {$x^+$};
	\node[right] at (2.7,0) {$x_{\mathrm{ext}}^+$};
	\node[right] at (0.5,2.1) {$\partial B_{1/2}$};
	\coordinate (Q) at ($(0,0) + (-45:2cm and 2cm)$);
	\draw[thick,magenta,-latex] ($(0,0) + (-45:2cm and 2cm)$(Q) arc (-45:45:2cm and 2cm);
	%
	\draw[thick,gray,-latex] ($(0,0) + (-45:2cm and 2cm)$(Q) arc (-45:-315:2cm and 2cm);
	\coordinate (P) at ($(1.41,0) + (-90:2cm and 1.41cm)$);
    \draw[thick, gray,-latex]($(1.41,0) + (-90:2cm and 1.41cm)$(P) arc (-90:89:2cm and 1.41cm);
	%
	\draw[thick, magenta,latex-] ($(1.41,0) + (90:2cm and 1.41cm)$(P) arc (90:269:2cm and 1.41cm);
	\end{tikzpicture}
	\\
	\begin{tikzpicture}[scale=0.8]
	\draw[dashed,gray] (0,0) circle (2cm);
	\draw (0,0) circle (2.2cm);
	\draw[dashed,gray] (8,0) circle (2cm);
	\draw (8,0) circle (2.2cm);
	%
	\draw (1.24,0) ellipse (2cm and 1.57cm);
	\draw (0.32,0) ellipse (2cm and 1.97cm);
	\draw (8-1.24,0) ellipse (2cm and 1.57cm);
	\draw (8-0.32,0) ellipse (2cm and 1.97cm);
	\fill (1.56,1.56) circle[radius=2pt];
	\fill (0,0)  circle[radius=2pt];
	\fill (1.56,-1.56)  circle[radius=2pt];
	\fill (8-1.56,1.56) circle[radius=2pt];
	\fill (8,0)  circle[radius=2pt];
	\fill (8-1.56,-1.56)  circle[radius=2pt];
	\node[right] at (0,0) {$0$};	
	\node[below] at (1.67,-1.56) {$p$};
	\node[above] at (1.67,1.56) {$q$};
	\node[right] at (8,0) {$0$};	
	\node[below] at (8-1.67,-1.56) {$p$};
	\node[above] at (8-1.67,1.56) {$q$};
	\node[right] at (-2.4,-2.3) {$\partial B_{1/2}$};
	\node[right] at (9.7,-2.3) {$\partial B_{1/2}$};
	\draw[->] (-2,-2)->(-1.4,-1.4);
	\draw[->] (10,-2)->(9.4,-1.4);
	\node[right] at (-1.7,0) {$x_{\mathrm{int},1}^-$};
	\node[right] at (-0.4,-0.8) {$x_{\mathrm{int},2}^-$};
	\node[right] at (2.25,0) {$x_{\mathrm{ext},1}^+$};
	\node[right] at (2.8,1) {$x_{\mathrm{ext},2}^+$};
	\node[right] at (8.8,0) {$x_{\mathrm{int},1}^+$};
	\node[right] at (7.5,-0.8) {$x_{\mathrm{int},2}^+$};
	\node[right] at (4.8,0) {$x_{\mathrm{ext},1}^-$};
	\node[right] at (4.2,1) {$x_{\mathrm{ext},2}^-$};
	\coordinate (Q) at ($(1.24,0) + (-80:2cm and 1.57cm)$);
	\draw[thick, gray,-latex] ($(1.24,0) + (-80:2cm and 1.57cm)$(Q) arc (-80:80:2cm and 1.57cm);
	\draw[thick, magenta,-latex] ($(1.24,0) + (-80:2cm and 1.57cm)$(Q) arc (-80:-277:2cm and 1.57cm);
	\coordinate (P) at ($(0.32,0) + (-50:2cm and 1.97cm)$);
	\draw[thick, magenta,-latex] ($(0.32,0) + (-50:2cm and 1.97cm)$(P) arc (-50:50:2cm and 1.97cm);
	\draw[thick, gray,-latex] ($(0.32,0) + (-50:2cm and 1.97cm)$(P) arc	(-50:-307:2cm and 1.97cm);
	\coordinate (QQ) at ($(8-1.24,0) + (-100:2cm and 1.57cm)$);
	\coordinate (PP) at ($(0.32,0) + (-50:2cm and 1.97cm)$);
	\draw[thick, gray,-latex] ($(8-1.24,0) + (-100:2cm and 1.57cm)$(QQ) arc (-100:-259:2cm and 1.57cm);
	\draw[thick, magenta,-latex] ($(8-0.32,0) + (-130:2cm and 1.97cm)$(PP) arc (-130:-229:2cm and 1.97cm);
	\draw[thick, magenta,-latex] ($(8-1.24,0) + (-100:2cm and 1.57cm)$(QQ) arc (-100:97:2cm and 1.57cm);
	\draw[thick, gray,-latex] ($(8-0.32,0) + (-130:2cm and 1.97cm)$(PP) arc (-130:127:2cm and 1.97cm);
\end{tikzpicture}
\end{tabular}
\end{table}

\begin{remark}\label{rem_casodegenere}
Notice that the case $\vert p \vert = \vert q \vert = 1/2$ and $\phi_0 = \pi/2$, giving rise to two semi-circular orbits, is left out from Theorem \ref{teo-var}. Indeed, the arguments used in the proof of Theorem \ref{teo-var} show that both these solutions are degenerate: $q$ is conjugate to $p$ along the corresponding geodesic. Hence, they have zero Morse index, but we cannot decide if they are local minimizers or not.   
\end{remark}

The remaining part of this section is devoted to the proof of Theorem \ref{teo-var}. For this, we will make use of the theory explained in the Appendix, based on geodesic bifurcation and the Morse index theorem.

\begin{remark}
It is worth mentioning that a slightly different variational formulation of the fixed-energy problem \eqref{eq-main} is possible.
Precisely, again up to a reparameterization ($\eta(s) = x(\tau(s))$), solutions to \eqref{eq-main} correspond to critical points of the Maupertuis-Jacobi functional
$$
J(\eta) = \int_0^1 \vert \dot \eta(s) \vert^2 \,ds \int_0^1 \left( \frac{1}{\vert \eta(s) \vert} - 1\right)\,ds,
$$
see for instance \cite{AmbCot93}. This variational formulation has the advantage that the parameterization is made through an affine change of variables (differently from \eqref{cambiost}, which is nonlinear): hence, the functional $J$ has a clearer cinematical interpretation with respect to the energy functional $E$.
Also, when adopting tools of Nonlinear Analysis in order to prove the existence of critical points, using $J$ is often simpler.
On the other hand, the energy functional $E$ has a neater geometrical meaning, since it is directly connected with the theory of geodesics in Riemannian manifolds: for this reason,
we have chosen to deal with it. It is not difficult, however, to see that our result can be rephrased in terms of the Maupertuis-Jacobi functional $J$.  
\end{remark}

\subsection{Antipodal points on ellipses}

In this section we describe a preliminary geometrical property of ellipses which will play a key role in our arguments.

\begin{definition}\label{def_antipodali}
Let $\mathcal{E} \subset \mathbb{R}^2$ be an ellipse with one focus at the origin and let $f$ be the other focus of $\mathcal{E}$.
Given $p \in \mathcal{E}$, we call antipodal point of $p$ the intersection point $p_f$ between the ellipse $\mathcal{E}$ and the straight line joining $p$ and $f$.  
\end{definition}

Note that, if $\mathcal{E}$ is a circumference, the point $p_f$ is simply the symmetric of $p$ with respect to the origin.

The result which we are going to prove is the following. 

\begin{proposition}\label{prop_antipodali}
Let $\mathcal{E} \subset \mathbb{R}^2$ be an ellipse with foci $0$ and $f = (x_f,0)$ and major axis of length $1$ and let $p \in \mathcal{E}$.
Let $\{\mathcal{E}_n\}_n \subset \mathbb{R}^2$ be a sequence of ellipses, with $\mathcal{E}_n \neq \mathcal{E}$ for any $n$, satisfying the following conditions:
\begin{itemize}
\item[i)] one focus of $\mathcal{E}_n$ is at the origin and the major axis has length $1$,
\item[ii)] $p \in \mathcal{E}_n$,
\item[iii)] the other focus $f_n$ of $\mathcal{E}_n$ converges to $f$, as $n \to +\infty$.
\end{itemize}
Then, for any $n$, the ellipses $\mathcal{E}_n$ and $\mathcal{E}$ intersect in a unique point $p_n$ different from $p$, and $p_n \to p_f$ as 
$n \to +\infty$.
\end{proposition}

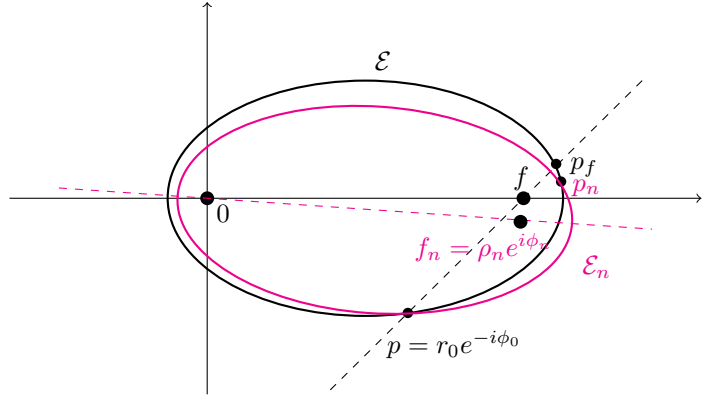
\begin{figure}[H]
	\centering
	\begin{tikzpicture}[scale=1.3]
	\draw[->] (-2,0) -- (5,0);
	\draw[->] (0,-2) -- (0,2);
	\draw[thick] (1.6,0) ellipse (2cm and 1.2cm);
	\node[right] at (1.6,1.4) {$\mathcal{E}$};
	\fill (0,0)  circle[radius=2pt];
	\node[right] at (0,-0.15) {$0$};
	\fill (3.2,0)  circle[radius=2pt];
	\node[right] at (3,0.2) {$f$};
	\fill (3.53,0.35)  circle[radius=1.5pt];
	\node[right] at (3.58,0.32) {$p_f$};
	\fill (3.58,0.17)  circle[radius=1.5pt];
	\node[right,magenta] at (3.6,0.11) {$p_n$};
	\fill (2.03,-1.17)  circle[radius=1.5pt];
	\node[below] at (2.5,-1.25) {$p=r_0e^{-i\phi_0}$};
	\node[right,magenta] at (3.7,-0.7) {$\mathcal{E}_n$};
	\draw[dashed] (4.4,1.2) -- (1.2,-2);
	\draw[rotate=-4,thick,magenta] (1.7,0) ellipse (2cm and 1.054cm);
	\draw[dashed,magenta] (-1.5,0.105) -- (4.5,-0.315);
	\fill (3.17,-0.24)  circle[radius=2pt];
	\node[right,magenta] at (2,-0.5) {$f_n=\rho_ne^{i\phi_n}$};
	\end{tikzpicture}
	\caption{A graphical explanation of Proposition \ref{prop_antipodali}. The ellipses $\mathcal{E}$ and $\mathcal{E}_n$ share one of the foci 
	(the origin), the length of the major axis and they both pass through $p$. While $\mathcal{E}_n$ is converging to $\mathcal{E}$, the intersection point $p_n$ between $\mathcal{E}$ and $\mathcal{E}_n$ converges to the point $p_f$, the antipodal point of $p$ with respect to the second focus $f$ of $\mathcal{E}$.}
	\label{figbifor}
\end{figure}

Before giving the proof, we show below that a sequence of ellipses with the above properties actually exists and we 
obtain their polar equations. Preliminarily, we observe that, by symmetry, we can suppose that $p$ lies in the lower half-plane, that is,
$p = r_0 e^{-i\phi_0}$ with
$\phi_0 \in [0,\pi]$.

Suppose first that $\mathcal{E}$ is not a circumference and, to fix the ideas, that $x_f > 0$; then, the polar equation of $\mathcal{E}$ reads as
\begin{equation}\label{eqellpartenza}
r=\dfrac{0.5(1-e^2)}{1-e\cos \phi},\qquad \phi \in [0,2\pi], 
\end{equation}
where (recalling that the eccentricity of an ellipse is the ratio between the focal distance and the length of the major axis) $e = x_f$.
Writing $f_n = \rho_n e^{i\phi_n}$, since from condition iii) it follows that $\phi_n \to 0$, we infer that for $n$ large enough 
the equation of $\mathcal{E}_n$ writes as
\begin{equation}\label{eqelln1}
r=\dfrac{0.5(1-e_n^2)}{1-e_n\cos (\phi + \phi_n)},\qquad \phi \in [0,2\pi],
\end{equation}
where, similarly as above, $e_n = \rho_n$.
By imposing $p \in \mathcal{E}_n$, we obtain the relation
$$
\cos(-\phi_0 + \phi_n) = \frac{2r_0 - 1 + e_n^2}{2 r_0 e_n},
$$
implying, again since $\phi_n \to 0$, 
\begin{equation}\label{thetan1}
\begin{cases}
\vspace{0.20cm}
\phi_n = \phi_0 - \arccos \frac{2r_0 - 1 + e_n^2}{2 r_0 e_n} & \mbox{ if } \phi_0 \notin \{0,\pi\} \\
\vspace{0.20cm}
\vert \phi_n \vert = \arccos \frac{2r_0 - 1 + e_n^2}{2 r_0 e_n} & \mbox{ if } \phi_0 = 0 \\
\vert \phi_n \vert = \arccos\left(-\frac{2r_0 - 1 + e_n^2}{2 r_0 e_n}\right) & \mbox{ if } \phi_0 = \pi. \\
\end{cases}
\end{equation}
Summing up, we have obtained that any ellipse of a sequence $\{\mathcal{E}_n\}_n$ satisfying conditions i), ii) and iii)
has polar equation given by \eqref{eqelln1}, where $e_n =\rho_n$, $\phi_n$ is given by \eqref{thetan1}, and $f_n = \rho_n e^{i\phi_n}$. 

Suppose now that $\mathcal{E}$ is a circumference, that is $r_0 = 1/2$. Then the equation for $\mathcal{E}_n$ writes again as \eqref{eqelln1} where, by condition iii), $e_n \to 0$. By imposing $p \in \mathcal{E}_n$ we find 
\begin{equation}\label{thetan2}
\phi_n = \phi_0 \pm \arccos e_n
\end{equation}
Notice that in both these cases we do not have $\phi_n \to 0$, differently from \eqref{thetan1}, indeed $\phi_n \to \phi_0 \pm \pi/2$.

With this in mind, we are ready to give the proof.

\begin{proof}[Proof of Proposition \ref{prop_antipodali}]
We first deal with the easier case when $\mathcal{E}$ is a circumference. In this situation, it is readily checked that
the intersection point $p_n = 0.5 e^{i\psi_n}$ between $\mathcal{E}$ and $\mathcal{E}_n$ is given by
{
$$
\psi_n = -\phi_n \pm \arccos e_n.
$$ 
Replacing equation \eqref{thetan2} and passing to the limit as $n \to +\infty$ we obtain that $\psi_n \to -\phi_0 \pm \pi$ (the case $\psi_n \to -\phi_0$ gives indeed the point $p$), hence $\phi_n$ in \eqref{thetan2}. Since $e_n \to 0$, 
$p_n$ converges to the antipodal point of $p$, as desired.
}

Suppose now that $\mathcal{E}$ is not a circumference. Again, we assume $x_f > 0$ (and $p$ lying in the lower half-plane).
Moreover, we also suppose that the eccentricity $e$ of $\mathcal{E}$ is given by
$$
e = r_0 \cos\phi_0 + \sqrt{r_0^2 \cos^2 \phi_0 + 1 - 2 r_0},
$$
(the case when $e = r_0 \cos\phi_0 - \sqrt{r_0^2 \cos^2 \phi_0 + 1 - 2 r_0}$ can be treated in the same way); accordingly, we have
$$
e_n = r_0 \cos (-\phi_0+\phi_n) + \sqrt{r_0^2 \cos^2 (-\phi_0+\phi_n) + 1 - 2 r_0}.
$$
Let us remark that, since $x_f=e$, the choice done on $e$ implies that $x_f$ is greater than the first component of $p$, hence the first component of $p_f$ is greater than $x_f$ and in particular is positive (see Figure \ref{figbifor}).

For the sake of readability, we split our arguments in three steps.
\smallbreak
\noindent{\emph{1. Finding the point $p_n = r_n e^{i\psi_n}$}.} This is easily done, solving the system made by equations 
\eqref{eqellpartenza} and \eqref{eqelln1}; we obtain that $\psi_n$ is the angle satisfying 
\begin{equation}\label{cossin}
\begin{cases}
\vspace{0.20cm}
\displaystyle{\cos\psi_n = \frac{-A_n C_n + \vert B_n \vert\sqrt{A_n^2 + B_n^2 - C_n^2}}{A_n^2 + B_n^2}}  \\
\displaystyle{\sin\psi_n = \frac{-A_n \vert B_n \vert \sqrt{A_n^2 + B_n^2 - C_n^2} - B_n^2 C_n}{B_n(A_n^2 + B_n^2)}}
\end{cases}
\end{equation}
where
$$
A_n = e(1-e_n^2) - e_n (1-e^2)\cos\phi_n, \qquad B_n = e_n(1-e^2)\sin\phi_n, \qquad C_n = e_n^2 - e^2.
$$
Of course, $r_n$ can be obtained from 
equations \eqref{eqellpartenza} or \eqref{eqelln1}.
\smallbreak
\noindent{\emph{2. Finding the point $p_* = \lim_{n \to +\infty}p_n$.}} By Taylor's expansions of $\cos \phi_n$ as $\phi_n \to 0$, we find
$$
e_n = e + (r_0\sin \phi_0 -E)\phi_n +F\phi_n^2+O(\phi_n^3)
$$
hence, as $n \to +\infty$, 
$$
\begin{cases}
\vspace{0.2cm}
A_n = (1+e^2)(E-r_0\sin\phi_0)\phi_n + \left( \frac12 e (1-e^2) - (1+e^2) F - e (E-r_0\sin\phi_0)^2 \right)\phi_n^2 +  
O(\phi_n^3) \\
\vspace{0.2cm}
B_n = e(1-e^2)\phi_n - (1-e^2)(E - r_0\sin\phi_0)\phi_n^2 + O(\phi_n^3), \\
C_n = -2e (E-r_0\sin\phi_0)\phi_n + \left((E-r_0\sin\phi_0)^2 + 2eF\right)\phi_n^2 + O(\phi_n^3), 
\end{cases}
$$
where
$$
E = -\frac{r_0^2\cos\phi_0\sin\phi_0}{\sqrt{\Delta}}, \quad 
F = \frac12 \left(- r_0\cos\phi_0 +  \frac{r_0^2 (\sin^2\phi_0 - \cos^2\phi_0)}{\sqrt{\Delta}} -  \frac{r_0^4 \cos^2\phi_0 \sin^2\phi_0}{\Delta^{3/2}}\right),
$$
with
$$
\Delta = r_0^2 \cos^2 \phi_0 + 1 - 2 r_0.
$$
Let us remark that, as discussed in the proof of Theorem \ref{teo-ex}, $\Delta \geq 0$.

Passing to the limit in \eqref{cossin}, we finally get $p_* = r_* e^{i\psi_*}$, where
$\psi_*$ is the angle satisfying
$$
\begin{cases}
\vspace{0.2cm}
\displaystyle{\cos\psi_* = \frac{U + V \sqrt{\Delta}}{Z}}  \\
\displaystyle{\sin\psi_* = -(1-e^2)r_0\sin\phi_0 \frac{-(1-r_0)(1+e^2)+2e\sqrt{\Delta}}{(1+e^2)^2 r_0^2 \sin^2\phi_0 + (1-e^2)\Delta}} \\ \vspace{.2cm}
\displaystyle{\qquad \quad = -\frac{r_0\sin\phi_0}{(1-e^2)Z\sqrt{\Delta}} \left[-(1+e^2)(U+V\sqrt{\Delta})+2eZ\right]},
\end{cases}
$$
where
$$
U = 2e(1+e^2)r_0^2 \sin^2\phi_0, \quad V = (1-r_0)(1-e^2)^2, \quad Z = (1+e^2) r_0^2 \sin^2\phi_0 + (1-e^2)^2 \Delta.
$$
Moreover,
$$
r_* = \frac{0.5(1-e^2)Z}{Z-e(U+V\sqrt{\Delta})}.
$$
\smallbreak
\noindent{\emph{3. Proving that $p_* = p_f$.}} In order to do this, we compute
the slopes $m$ and $m_*$ of the lines through $p$ and $f$ and through $p_*$ and $f$, respectively,  and show that they are equal.
Recalling that $x_f = e$, we find
$$
m = \frac{r_0\sin\phi_0}{\sqrt{\Delta}} \quad \mbox{ and } \quad m_* = \frac{-r_0\sin\phi_0 \left( 2eZ - (1+e^2)(U + V \sqrt{\Delta})\right)}{(1-e^2)(U+V\sqrt{\Delta}) - 2 e Z + 2e^2 (U+V\sqrt{\Delta})}\frac{1}{\sqrt{\Delta}}
$$
and the conclusion follows by simple computations.
\end{proof}

\subsection{From antipodal points to bifurcation points}\label{sec3.2}

We now show an interpretation of Proposition \ref{prop_antipodali} in terms of geodesic bifurcation.
From now on, $\gamma_*: [0,1] \to \mathcal{H}$ will be a geodesic parameterization of a Keplerian ellipse of energy $h$, 
with $\gamma_*(0) = \gamma_*(1) = p$ and $\gamma_*(s) \neq p$ for any $s \in \mathopen{]}0,1\mathclose{[}$ (notice that we are not precising if it is covered in the clockwise/counter-clockwise sense).

\begin{proposition}\label{prop_biforcazione}
The only bifurcation point along $\gamma_*$ is the point $p_f$ introduced in Definition \ref{def_antipodali}.
\end{proposition}

\begin{proof}
We divide the proof in two parts.
\smallbreak
We first prove that $p_f$ is a bifurcation point. Let $\mathcal{E} = \gamma_*([0,1])$: by Remark \ref{rem_geoellisse}, $\mathcal{E}$ is an ellipse with focus at the origin and major axis of length $1$. Up to a rotation, we can assume that the second focus $f$ is of the form $(x_f,0)$ for some $x_f \in \mathbb{R}$.

Let us now construct a sequence of ellipses 
$\{\mathcal{E}_n\}_n$ as in Proposition \ref{prop_antipodali}. By that result, the ellipses $\mathcal{E}$ and $\mathcal{E}_n$ intersect in a unique point different from $p$, say $p_n$, satisfying $p_n \to p_f$ as $n \to +\infty$.
Set now $s^* = \gamma^{-1}_*(p_f)$ and $s_n = \gamma^{-1}_*(p_n)$ and, according to Remark \ref{rem_geoellisse} and the second property in Proposition \ref{prop_geo}, let $\gamma_n: [0,1] \to \mathcal{H}$ be a geodesic parameterization 
of (an arc of) $\mathcal{E}_n$ such that
$$
\gamma_n(0) = p, \qquad \gamma_n(s_n) = p_n, \qquad \langle \dot\gamma_n(0),\dot\gamma_*(0) \rangle \geq 0.
$$
We claim that the above $\gamma_n$ satisfy the conditions in Proposition \ref{prop_bifo_astratta} for the geodesic bifurcation.
Of course, $\gamma_n \neq \gamma$ and $s_n \to s^*$; moreover, by construction conditions i) and ii) of Proposition \ref{prop_bifo_astratta}
are satisfied. As far as iii) is concerned, by geometrical considerations we easily infer that there exists $\lambda \geq 0$
$$
\dot \gamma_n(0) \to  \lambda \dot \gamma_*(0).
$$
We thus need to show that $\lambda = 1$.

By the continuous dependence theorem for ODEs, we deduce that $\gamma_n \to \tilde \gamma$, uniformly on $[0,1]$, where
$\tilde \gamma$ is the unique geodesic with $\tilde\gamma(0) = p$ and $\dot{\tilde\gamma}(0) = \lambda \dot \gamma_*(0)$.
Thus, either $\tilde \gamma$ is constantly equal to $p$ (if $\lambda = 0$) or, by the Rescaling Lemma in Proposition \ref{prop_geo}, 
$\tilde\gamma(s) = \gamma_*(\lambda s)$ for every $s$ (if $\lambda > 0$).
Passing to the limit in $\gamma_n(s_n) = p_n$, using the fact that $p_n \to p_f$ in view of Proposition \ref{prop_antipodali},
we get $\tilde \gamma(s^*) = p_f$. This of course excludes that $\lambda = 0$; hence $\tilde\gamma(s^*) = p_f = \gamma_*(\lambda s^*)$,
implying $\lambda = 1$ since $p_f = \gamma_*(s^*)$.

\smallbreak
We now show that $p_f$ is the only bifurcation point. By contradiction, let us assume that $p' \neq p_f$ is another bifurcation point along $\gamma$ and let $\{\gamma_n\}_n$ be a family of bifurcation geodesics according to Definition \ref{def_bifo}. According to Remark \ref{rem_geoellisse}, $\gamma_n([0,1])$ are arcs of ellipses with one focus at the origin and major axis of length $1$; moreover, since $\gamma_n \to \gamma$ it is easily checked that such ellipses are of the type described in Proposition \ref{prop_antipodali}. By the notion of geodesic bifurcation, we further know that $\gamma_n([0,1])$ intersect $\gamma([0,1])$ in a point different from $p$ and converging to $p'$. Then, Proposition \ref{prop_antipodali} finally implies $p' = p_f$.
\end{proof}

\subsection{Proof of Theorem \ref{teo-var}}

As in the proof of Theorem \ref{teo-ex}, we assume (by rotational invariance) that $p$ and $q$ are symmetric with respect to the $x$-axis and that the point $p$ belongs to the lower half-plane.

Let us denote by $x$ a solution of \eqref{eq-main} and recall that the local minimality/non-minimality of $x$ is meant as the local minimality/non-miminality of the corresponding geodesic $\gamma$ as a critical point of the energy functional $E$. 
Denoting by $\mu(\gamma)$ the Morse index of $\gamma$, by Taylor's formula the following hold true:
\begin{itemize}
\item if $\mu(\gamma) \neq 0$, then $\gamma$ is not a local minimizer of $E$,
\item if $\mu(\gamma) = 0$ and $\gamma$ is non-degenerate (that is, if $\gamma(1)$ is not conjugate to $\gamma(0) = p$ along $\gamma$), then
$\gamma$ is a local minimizer of $E$.
\end{itemize}
In view of the Morse index theorem (Theorem \ref{teo_morse}) we thus have:
\begin{itemize}
\item if there exists $s^* \in \mathopen{]0},1\mathclose{[}$ such that $\gamma(s^*)$ is conjugate to $p$ along $\gamma$, then $\gamma$ is not a local minimizer of $E$,
\item if for any $s \in \mathopen{]}0,1\mathclose{]}$ the point $\gamma(s)$ is not conjugate to $p$ along $\gamma$,
then $\gamma$ is a (non-degenerate) local minimizer of $E$.
\end{itemize}
We thus need to count conjugate points along $\gamma$. To this end, we regard the Keplerian arc $\gamma$ as part of the full Keplerian ellipse; that is, we consider the geodesic $\gamma_*$ defined at the beginning of the previous Section \ref{sec3.2}: of course, we take the parameterization in the clockwise/counter-clockwise direction according to the one of the arc $x$ we are dealing with. 
Needless to say, points of (the image of) $\gamma$ are conjugate to $p$ along $\gamma$ if and only if they are conjugate to $p$ along $\gamma_*$.

By Theorem \ref{prop_bifo_astratta}, any point of the Keplerian ellipse (parameterized by) $\gamma_*$ different from $p$ itself is conjugate to $p$ if and only if it is a bifurcation point along $\gamma_*$. Using Proposition \ref{prop_biforcazione}, we thus know that the only point 
(different from $p$) conjugate to $p$ along the ellipse is the antipodal point $p_f$. Let us also notice that $q \neq p_f$, since we have excluded the case $\vert p \vert = 1/2$ and $\phi_0 = \pi/2$ (see Remark \ref{rem_casodegenere}). As a consequence of this discussion, we have that $x$ is a local minimizer if and only if $p_f \notin x([0,1])$.

To understand when this happens, we assume without loss of generality that $\phi_0 < \pi/2$; then, by simple geometrical considerations, the following holds: denoted by $f = (x_f,0)$ the second focus of the Keplerian ellipse containing the Keplerian arc $x$, if $x_f < \vert p \vert \cos \phi_0$, then $p_f \in x([0,1])$ if and only if $c_x < 0$;  
if $x_f > \vert p \vert \cos \phi_0$, then $p_f \in x([0,1])$ if and only if $c_x > 0$. Since the eccentricity of an ellipse is the ratio between the focal distance and the length of the major axis, $\vert x_f \vert = e$, and going back to the proof of Theorem \ref{teo-ex}, we have the following:
\begin{enumerate}
\item Case $|p|=|q|<1/2$. If $x$ belongs to the ellipse with focus on the negative $x$-semiaxis, then $x$ is a local minimizer if and only if it satisfies $c_x > 0$, that is if and only if $x = x_{\mathrm{int}}^+$; if $x$ belongs to the ellipse with focus on the positive $x$-semiaxis, then
by \eqref{eq-eccelldentro2} it holds $x_f = e > \vert p \vert \cos \phi_0$ and then $x$ is a local minimizer if and only if it satisfies $c_x < 0$, that is if and only if $x = x_{\mathrm{int}}^-$.
\item Case $|p|=|q|=1/2$. If $x$ belongs to the circular orbit, we easily conclude. Otherwise, $x$ belongs to an ellipse with focus on the positive $x$-semiaxis and, by \eqref{eq-ellissecirc1}, $x_f = e = \cos\phi_0 > \vert p \vert \cos\phi_0$; again, we infer that $x$ is a local minimizer if and only if $c_x < 0$, that is if and only if $x = x_{\mathrm{int}}$. 
\item Case $1/2<|p|=|q|<1$. Of course, since we are assuming $\phi_0 < \pi/2$, we have $\phi_0 < \arcsin \frac{1-|p|}{|p|}$, as well.
Then, by \eqref{eq-ellissefuoricirc2}, $x$ belongs to an ellipse with focus on the positive $x$-semiaxis: one of this ellipse
satisfies $x_f = e > \vert p \vert \cos \phi_0$, while the other one satisfies $x_f = e < \vert p \vert \cos \phi_0$.
Thus, $x$ is a local minimizer if and only if $c_x < 0$ in the first case and $c_x > 0$ in the second one. Since the ellipse with greater eccentricity is the one with farthest perihelion, the solution $x$ is a local minimizer if and only if $x= x_{\mathrm{ext},1}^+$ or $x = x_{\mathrm{int},2}^-$. 
\end{enumerate}

\begin{remark}\label{indici_esatti}
Let us observe that the above arguments do not allow us to compute the Morse index of an arc $x$ which is not locally minimal: indeed, we have shown that there is only one conjugate point along $x$, but we do not have information about its multiplicity. We conjecture, however, that the Morse index of all these arcs should be exactly equal to $1$. 
\end{remark}

\begin{remark}
It is interesting to observe that variational characterization provided by Theorem \ref{teo-var} is no longer valid if the Kepler equation 
is considered in the three-dimensional space. Indeed, a further geodesic bifurcation can be produced, leading to new conjugate points.
Precisely, by rotating a Keplerian arc around the line passing through the origin and the point $p$, it is immediately checked that $p'$, the symmetric point of $p$ with respect to the origin, is a bifurcation point (hence, a conjugate point) as well.
Hence, for the Kepler equation in $\mathbb{R}^3$, only one of the solutions of \eqref{eq-main} is locally minimal.  
We omit the precise statement of the result, for briefness. 
\end{remark}

\section{Appendix: a brief recap on geodesics and their Morse index}\label{sec 4}

In this section, we collect some known facts about geodesics and conjugate points. 
We will work in the simplified setting when the manifold is an open set of the Euclidean space, endowed with a Riemannian metric conformal to the standard one. For a more general treatment, we refer to \cite{Kli95}.

\subsection{Definition and first properties}

Let $M$ be an open subset of $\mathbb{R}^N$; as such, $M$ has a natural structure of differentiable manifold of dimension $N$. 
Given a smooth function $W: M \to \mathbb{R}$ satisfying $W(x) > 0$ for every $x \in M$, we can endow $M$ with the Riemannian structure
given by 
$$
g_W(x)[u,v] = W(x) \langle u,v \rangle, \qquad x \in M, \,u,v \in \mathbb{R}^N,
$$
where $\langle \cdot, \cdot \rangle$ denotes the Euclidean product on $\mathbb{R}^N$. Notice that smooth curves on $M$ are nothing but curves with values in $M$ which are smooth in the usual sense; however, the velocity and the length of a curve are influenced by the function $W$.  

A \emph{geodesic} on the Riemannian manifold $(M,g_W)$ is a smooth curve
$\gamma: I \to M$ (with $I \subset \mathbb{R}$ an interval) solving the equation
\begin{equation}\label{eq_geo}
\frac{d}{ds} \left( W(\gamma) \dot \gamma \right) = \frac12 \vert \dot \gamma \vert^2 \nabla W(\gamma), 
\end{equation}
where $\nabla W$ is the gradient of the function $W$ with respect to the standard Euclidean structure. 
By the standard theory of ODEs, for any $p \in M$ and $v \in \mathbb{R}^N$, there exists a unique (maximal) geodesic $\gamma_{p,v}$ 
satisfying $\gamma_{p,v}(0) = p$ and $\dot \gamma_{p,v}(0) = v$.

We collect in the next proposition some simple properties which can be readily obtained from this definition.

\begin{proposition}\label{prop_geo}
	The following facts hold true:
	\begin{itemize}
		\item[i)] Geodesics have constant velocity: namely, for any geodesic $\gamma$, there exists $c \in \mathbb{R}$ such that $\vert \dot \gamma(s) \vert \sqrt{W(\gamma(s))} \equiv c$;
		\item[ii)] For every geodesic $\gamma$ and for every $a,b \in \mathbb{R}$ with $a \neq 0$, the curve
		$\tilde\gamma(s) = \gamma(as + b)$ is a geodesic, as well;
		\item[iii)] (Rescaling Lemma) For any $\lambda \in \mathbb{R}$ with $\lambda \neq 0$, it holds that $\gamma_{p,\lambda v}(s) = \gamma_{p,v}(\lambda s)$ for any $s \in \mathbb{R}$ such that both sides of the equation are defined.
	\end{itemize}
\end{proposition}

In view of this proposition, geodesics are often considered as geometric objects, rather than as parameterized curves. In some textbooks, curves which can be parameterized as geodesics are called pre-geodesics. 

\subsection{The energy functional and the Morse index theorem}

Equation \eqref{eq_geo} can be meant as the Euler-Lagrange equation of the Lagrangian
$L(\gamma,\dot\gamma) = \vert \dot \gamma \vert^2 W(\gamma)$; thus, geodesics can be regarded as critical points of the functional
$E(\gamma) = \int L(\gamma,\dot \gamma) \,ds$. Below, we describe in more details this variational principle, by focusing on the case of geodesics joining two points $p,q \in M$ (a similar discussion can be made when periodic boundary conditions are taken into account, leading to the so-called closed geodesics). 

Let us consider the functional 
$$
E(\gamma) = \int_0^1 \vert \dot \gamma(s)\vert^2 W(\gamma(s))\,ds,
$$
defined on the Hilbert manifold $\Omega_{p,q}$ of $H^1$-paths $\gamma: [0,1] \to M$ satisfying $\gamma(0) = p$ and $\gamma(1) = q$.
It is standard to verify that $E$ is a smooth functional on $\Omega_{p,q}$ with first Frechet differential given by 
$$
dE(\gamma)[\xi] = \int_0^1 \left( 2 \langle \dot\gamma(s),\dot\xi(s)\rangle W(\gamma(s)) + \vert \dot \gamma(s) \vert^2 \langle \nabla W(\gamma(s)),\xi(s)\rangle \right) \,ds,
$$
for any $\xi \in H^1_0([0,1];\mathbb{R}^2)$. A well-known argument based on integration by parts yields:

\begin{proposition}\label{prop_energia}
	A curve $\gamma: [0,1] \to M$ satisfying $\gamma(0) = p$ and $\gamma(1) = q$ is a geodesic if and only if it is a critical point of the functional $E$.
\end{proposition}

We are now in a position to define the crucial concept of this section. By \emph{Morse index} $\mu(\gamma)$ of a geodesic $\gamma$ we mean its Morse index as a critical point of the functional $E$, that is, the dimension of the maximal subspace $V \subset H^1_0([0,1];\mathbb{R}^2)$ such that $d^2 E(\gamma)[\xi,\xi] < 0$ for any $\xi \in V$. 
Since
\begin{align*}
d^2  E(\gamma)[\xi,\xi] = & 2\int_0^1 \vert \dot\xi(s)\vert^2 W(\gamma(s))\,ds + 4 \int_0^1 \langle \dot\gamma(s),\dot \xi(s) \rangle \langle \nabla W(\gamma(s),\xi(s)\rangle\,ds \\ & + \int_0^1 \vert \dot \gamma(s) \vert^2 \langle \nabla^2 W(\gamma(s))\xi(s),\xi(s)\rangle \,ds
\end{align*}
for any $\xi \in H^1_0([0,1];\mathbb{R}^2)$, functions $\xi$ lying in the kernel of $d^2 E(\gamma)$ are nothing but solutions of the linear equation
\begin{equation}\label{eq-lin}
\frac{d}{ds}\Big(  W(\gamma(s)) \dot \xi  + \langle \nabla W(\gamma(s),\xi\rangle \dot\gamma(s) \Big) = \frac12 \vert \dot \gamma(s) \vert^2 \nabla^2 W(\gamma(s))\xi + \langle \dot\gamma(s),\dot \xi \rangle \nabla W(\gamma(s))
\end{equation}
satisfying the boundary condition $\xi(0) = \xi(1) = 0$. 

The celebrated Morse index theorem asserts that the Morse index $\mu(\gamma)$ equals the number of some special points along $\gamma$, the so-called \emph{conjugate points}. Precisely, we say that the point $\gamma(s^*) = p_{s^*}$, with $s^* \in \mathopen{]}0,1\mathclose{]}$, is conjugate to $\gamma(0) = p$ along $\gamma$ if the linear equation \eqref{eq-lin}
admits nontrivial solutions $\xi$ satisfying $\xi(0) = \xi(s^*) = 0$.
The multiplicity $m(p_{s^*})$ of the conjugate point is, by definition, the dimension of the space of such solutions.
Incidentally, notice that $\gamma(1) = q$ is conjugate to $\gamma(0)$ along $\gamma$ if and only if the kernel of $d^2 E(\gamma)$ is not trivial (that is, if and only if $\gamma$ is a degenerate critical point of $E$).

With this in mind, the Morse index theorem reads as follows (see for instance \cite{PorWat15} and the references therein).

\begin{theorem}[Morse Index Theorem]\label{teo_morse}
	Let $\gamma: [0,1] \to M$ be a non-constant geodesic. Then the set of conjugate points along $\gamma$ is finite and
	$$
	\mu(\gamma) = \sum_{s^* \in \,\mathopen{]}0,1\mathclose{[}} m(p_{s^*}).
	$$
\end{theorem}

\begin{remark}\label{rem_paramet}
The choice of parameterizing the geodesic $\gamma$ on the interval $[0,1]$ is completely conventional: any other interval could be used, since geodesics are preserved by affine reparameterizations (recall Proposition \ref{prop_geo}). Of course, conjugate points and the Morse index do not depend on this choice.
\end{remark}

\subsection{Bifurcation of geodesics}

A key role in our arguments will be played by the notion of geodesic bifurcation, as introduced in the paper \cite{PicPorTau04}.

\begin{definition}\label{def_bifo}
	Let $\gamma: [0,1] \to M$ be a non-constant geodesic. A point $\gamma(s^*)$, with $s^* \in \mathopen{]}0,1\mathclose{[}$, is a bifurcation point
	along $\gamma$ if there exists a sequence $\{s_n\}_n \subset [0,1]$ with $s_n \to s^*$ and a sequence of geodesics $\{\gamma_n\}_n$ (defined on $[0,1]$), with $\gamma_n \neq \gamma$, such that
	\begin{itemize}
		\item[i)] $\gamma_n(0) = \gamma(0)$, for every $n$,
		\item[ii)] $\gamma_n(s_n) = \gamma(s_n)$, for every $n$,
		\item[iii)] $\dot \gamma_n(0) \to \dot \gamma(0)$, for $n \to +\infty$.
	\end{itemize}
\end{definition}

According to the discussion in \cite[p. 122]{PicPorTau04} together with \cite[Corollary 5.6]{PicPorTau04}, the relationship between conjugate points and bifurcation points can be stated as follows.

\begin{theorem}\label{prop_bifo_astratta}
Let $\gamma: [0,1] \to M$ be a non-constant geodesic. Then, a point $\gamma(s^*)$, with $s^* \in \mathopen{]}0,1\mathclose{[}$, is a bifurcation point if and only if it is conjugate to $\gamma(0)$ along $\gamma$.
\end{theorem}



\begin{thebibliography}{99}

\footnotesize 

\bibitem{AmbCot93}
{A. Ambrosetti and V. Coti Zelati, 
Periodic solutions of singular Lagrangian systems,
Progress in Nonlinear Differential Equations and their
              Applications \textbf{10},
Birkh\"auser Boston, Inc., Boston, MA (1993)}

\bibitem{BahRab89}
{A. Bahri and P.H. Rabinowitz, \emph{A minimax method for a class of {H}amiltonian systems with
singular potentials}, J. Funct. Anal. \textbf{82} (1989), 412--428.}

\bibitem{BaJaPo}
{V.L. Barutello, R.D. Jadanza and A. Portaluri, \emph{Morse index and linear stability of the Lagrangian circular orbit in a three-body-type problem via index theory}, Arch. Ration. Mech. Anal. \textbf{219} (2016), 387--444.} 

\bibitem{BarTerVer14}
{V. Barutello, S. Terracini and G. Verzini, \emph{Entire parabolic trajectories as minimal phase transitions},
Calc. Var. Partial Differential Equations \textbf{49} (2014), 391--429.}

\bibitem{BosDamTer17}
{A. Boscaggin, W. Dambrosio and S. Terracini, \emph{Scattering parabolic solutions for the spatial $N$-centre problem}, Arch. Rat. Mech. Anal. \textbf{223} (2017), 1269--1306.}

\bibitem{Che}
{K.-C. Chen, \emph{Existence and minimizing properties of retrograde orbits to the three-body problem with various choices of masses}, 
Ann. of Math. (2) \textbf{167} (2008), 325--348.} 

\bibitem{CheMon}
{A. Chenciner and R. Montgomery, \emph{A remarkable periodic solution of the three-body problem in the case of equal masses},
Ann. of Math. (2) \textbf{152} (2000), 881--901.}
 
\bibitem{FerTer}
{D.L. Ferrario and S. Terracini, \emph{On the existence of collisionless equivariant minimizers for the classical n-body problem}, Invent. Math. \textbf{155} (2004), 305--362.}

\bibitem{FuGrNe}
{G. Fusco, G.F. Gronchi and P.~Negrini, \emph{Platonic polyhedra,
topological constraints and periodic solutions of the classical {$N$}-body
problem}, Invent. Math. \textbf{185} (2011), 283--332.}

\bibitem{Ge16}
{H. Geiges, {The Geometry of Celestial Mechanics}, London Mathematical Society Student Texts, Cambridge University Press, Cambridge, 2016.}

\bibitem{Gor77}
{W.B. Gordon, \emph{A minimizing property of Keplerian orbits},  Amer. J. Math. \textbf{99} (1977), 961--971.}

\bibitem{Gut90}
{M.C. Gutzwiller, Chaos in classical and quantum mechanics, Interdisciplinary Applied Mathematics \textbf{1}, Springer-Verlag, New York, 1990.}

\bibitem{HuSun10}
{X. Hu, S. Sun, \emph{Morse index and stability of elliptic Lagrangian solutions 
		in the planar three-body problem}, Adv. Math. (1) \textbf{223} (2010), 98-–119}.  

\bibitem{Ja}
{C.G.J. Jacobi, \emph{Zur {T}heorie der {V}ariations-{R}echnung und der
              {D}ifferential-{G}leichungen}, J. Reine Angew. Math., \textbf{17} (1837), 68--82}.

\bibitem{KOP}
{H. Kavle, D. Offin, A. Portaluri}, \emph{Keplerian orbits through the Conley-Zehnder index}, arXiv:1908.00075v1.

\bibitem{Kli95}
{W.P.A. Klingenberg, Riemannian geometry. Second edition, De Gruyter Studies in Mathematics \textbf{1}, Walter de Gruyter \& Co., Berlin, 1995.}

\bibitem{MadVen09}
{E. Maderna and A. Venturelli, \emph{Globally minimizing parabolic motions in the Newtonian $N$-body problem}, Arch. Ration. Mech. Anal. \textbf{194} (2009), 283--313.}

\bibitem{Mont}
{R. Montgomery, \emph{Minimizers for the Kepler problem}, preprint}.

\bibitem{OrUr10}
{R. Ortega and A.J. Ure\~na, Introducci\'on a la Mec\'anica Celeste,
Universidad de Granada,	Granada, 2010}.


\bibitem{PicPorTau04}
{P. Piccione, A. Portaluri and D.V. Tausk \emph{Spectral flow, Maslov index and bifurcation of semi-Riemannian geodesics}, 
Ann. Global Anal. Geom. \textbf{25} (2004), 121--149.}

\bibitem{PorWat15}
{A. Portaluri and N. Waterstraat, \emph{Yet another proof of the Morse index theorem}, Expo. Math. \textbf{33} (2015), 378--386.}

\bibitem{SoaTer12}
{N. Soave and S. Terracini, \emph{Symbolic dynamics for the {$N$}-centre problem at negative energies},
Discrete Contin. Dyn. Syst. \textbf{32} (2012), 3245--3301.}

\bibitem{TeVe}
{S. Terracini and A. Venturelli, \emph{Symmetric trajectories for the $2N$-body problem with  equal masses},  Arch. Rat. Mech. Anal. \textbf{184} (2007), 465--493.}

\bibitem{Tod}
{M.A. Todhunter, \emph{Researches in the Calculus of Variation, principally on the theory of discontinuous solutions}, Cambridge, Macmillan and co., 1871}.

\bibitem{Wint}
{A. Wintner, \emph{Analytical foundation of Celestial Mechanics}, Princeton Mathematical series, vo. 5, Princeton U. Press, 1947}.

\bibitem{Yu}
{G. Yu, \emph{Simple choreographies of the planar Newtonian $N$-body problem}, Arch. Ration. Mech. Anal. \textbf{225} (2017), 901--935.}


\end{thebibliography}
\end{document}